\documentclass[11pt,reqno]{article}
\usepackage{amsgen, amsmath, amsfonts, amsthm, amssymb, enumerate,amscd,graphics,dsfont,comment,tikz-cd}

\newtheorem{defn}{Definition}[section]
\newtheorem{prop}[defn]{Proposition}
\newtheorem{lem}[defn]{Lemma}
\newtheorem{thm}[defn]{Theorem}
\newtheorem{cor}[defn]{Corollary}

\newcommand {\ZZ}{{\mathds Z}}

\newcommand {\C}{{\mathds C}}

\newcommand {\Z}{{\mathcal Z}}
\newcommand {\CC}{{\mathcal C}}
\newcommand {\Q}{{\mathds Q}}
\newcommand {\QQ}{{\mathcal Q}}

\newcommand {\M}{{\mathcal M}}

\newcommand {\CP}{{\mathds P}}

\newcommand {\DD}{{\mathcal D}}

\newcommand {\TCC}{\tilde{\mathcal C}}

\newcommand {\TK}{{\tilde{K}}}

\newcommand {\TD}{{\tilde{D}}}

\newcommand {\tZ}{{\mathcal Z}}

\newcommand {\TDD}{{\tilde{\DD}}}
\newcommand {\TRR}{{\tilde{R}}}

\def\End{\operatorname{End}}

\def\div{\operatorname{div}}

\def\Im{\operatorname{Im}}
\def\ord{\operatorname{ord}}

\def\Hom{\operatorname{Hom}}
\def\Ext{\operatorname{Ext}}

\title{Old and new motivic cycles on Abelian surfaces}
\author{Ramesh Sreekantan}

\begin{document}
	
	\maketitle
	\baselineskip=17pt
	\begin{abstract}
		Collino \cite{colo} discovered indecomposable motivic cycles in the group $H^{2g-1}_{\M}(J(C),\ZZ(g))$. In \cite{sree22-1} we described the construction of some new motivic cycles which can be viewed as a generalization of Collino's cycle when $g=2$. In this paper we show that our new cycles are in fact related to Collino's cycles of higher genus. On one hand this suggests that new cycles are hard to find. On the other, it suggests that the tools developed to study Collino's cycle can be applied to our cycles. 
	\end{abstract}

\section{Introduction}

\subsection{Algebraic Cycles}

Let $(C,P$) be a pointed curve of genus $g\geq 3$ which is not hyperelliptic. A very important cycle, at least from the point of view of algebraic cycles, is the {\bf Ceresa cycle} in the Jacobian of {\bf non-hyperelliptic  curves} of genus $g>3$ given by 
$$Z_P=C_P-(-1)^*(C_P)$$
where $C_P$ is the image of $C$ under the map $\iota_P(x)=x-P$. This is a null homologous cycle lying in the group $CH^{g-1}_{hom}(J(C))$. In general it is {\bf not} algebraically equivalent to $0$ and is hence a non-trivial element of the Griffiths group of $J(C)$ \cite{cere}. 

Collino \cite{coll} discovered a higher Chow cycle $Z_{P_1,P_2,\tilde{R}}$ in  the group $CH^{g}(J(C),1)$ where $J(C)$ is  the Jacobian of a {\bf hyperelliptic curve}. This  can be viewed as a degeneration of the Ceresa cycle and depends on the curve  $C$ along with a function with divisor supported on two Weierstrass points $P_1$ and $P_2$. To determine the function precisely we have to choose a third point $\tilde{R}$ where the function takes the value $1$ for which it is convenient to choose another Weierstrass point.  Analogous to Ceresa's result,  it provides an example of a generically {\bf indecomposable cycle}, which roughly means that is not the product of cycles in other Chow groups. 
	
In the case of genus $2$, when all curves are hyperelliptic, Collino's element is defined in fibres over the complement of the moduli of products of elliptic curves in the moduli of Abelian surfaces.  In \cite{sree22-1} we generalised Collino's cycles to cycles defined in the fibres over the complement of certain {\bf Noether-Lefschetz divisors},  which are surfaces on the modular threefold  where the Neron-Severi group of the corresponding Abelian surface has an element which is not a multiple of the class of principal polarization. These divisors can also be understood as components of the image  of the  moduli of Abelian surfaces with real multiplication or moduli of products of elliptic curves.  These are  also called  {\bf Humbert surfaces} or {\bf Heegner divisors}. Collino's cycle is defined in the complement of the Humbert surface of invariant $1$.

There is a classical result of Kummer which gives a bijective correspondence between Abelian surfaces on the one hand and configurations of six  lines in $\CP^2$ tangent to a conic on the other. Humbert, and more recently Birkenhake-Wilhelm, related the existence of rational curves which have exceptional intersection with this configuration of lines to the corresponding Abelian surface having exceptional divisors, or equivalently, the moduli point lying on a Noether-Lefschetz divisor. 

In this paper we recall the construction of the new  motivic  cycles which are defined in the complement of such Noether-Lefschetz divisors. The main theorem states that in fact these cycles are related to Collino's cycles in higher genus hyperelliptic curves.  

\begin{thm} Let $\QQ_0$ be a rational curve in $\CP^2$ which determines a certain Noether-Lefschetz divisor as in Secton \ref{mainconstruction}. Let $A$ be an Abelian surface whose moduli point does not lie on that divisor.  Let $\Z_{\QQ_0,P,R}$ be the cycle in $H^3_{\M}(A,\Q(2))$ determined by $\QQ_0$, a node $P$  and an auxilliary point $R$ as in Theorem \ref{construction}. Then there is a hyperellipic curve $\TD$ mapping to a curve $D$ on $A$ and points $P_1,P_2$ and $\tilde{R}$ on $\TD$ such that under the induced map from the Jacobian 
$\pi:J(\TD)\rightarrow A$
	$$\pi_*(Z_{P_1,P_2,\tilde{R}})=2 \Z_{\QQ_0,P,R}$$
	where $Z_{P_1,P_2,\tilde{R}}$ is the cycle constructed by Collino in $CH^{g}(J(\TD),1)=H^{2g-1}_{\M}(J(\TD),\ZZ(g))$  as in section \ref{collino}.
	\end{thm}

On one hand, the theorem  shows that perhaps there are not too many ways to construct algebraic cycles. Two seemingly unrelated constructions lead to the essentially the same cycle. On the other hand it allows one to relate the new cycle to certain extensions of the mixed Hodge structure on the fundamental group of a curve. It is known that Abel Jacobi image of the Ceresa cycle \cite{hain,pult,kaen,rabi} and the regulator of the  Collino cycle \cite{colo,sasr} are related to explicit extensions of the mixed Hodge structure on the fundamental group of the curve. So one can similarly relate the regulator of the new cycles to  extensions of the the fundamental group.

\section{Motivic Cycles}

In this section we introduce the background on motivic cycles and Collino's cycle. 

\subsection{Motivic Cohomology Cycles}
\label{motiviccycles}

Let $X$ be a smooth projective algebraic variety of dimension $g$  defined over $\C$. The motivic cohomology group $H^{2g-1}_{\M}(X,\ZZ(g))$  has the following presentation: Generators  are represented by finite sums 
$$Z=\sum_i (C_i,f_i)$$
where $C_i$ are curves  on $X$ and $f_i:C_i \longrightarrow \CP^1$ are functions on them subject to the co-cycle condition 
$$\sum_i \div(f_i)=0 $$
Relations in this group are defined as follows. If $Y$ is a surface on $X$ and  $f$ and $g$ are functions on $Y$, one has the Steinberg element  $\{f,g\}$ in $K_2(\C(Y))$, where $\C(Y)$ is the function field of $Y$. To such an element one can consider the sum, called the tame symbol of $\{f,g\}$, 
$$\tau(\{f,g\})=\sum_{W\in Y^{(1)}} (W,(-1)^{\ord_W(f)\ord_W(g)}\frac{f^{\ord_W(g)}}{g^{\ord_W(f)}}).$$
where $Y^{(1)}$ is the  collection of curves on $Y$. This is a finite sum and satisfies the co-cycle condition, hence lies in the above group.  An element is said to be $0$ in $H^{2g-1}_{\M}(X,\ZZ(g))$  if it lies in the image of the free abelian group generated by the tame symbols of elements of $K_2(\C(Y))$ for some surface $Y \subset X$. The group $H^{2g-1}_{\M}(X,\ZZ(g))$ is the same as the higher Chow group $CH^g(X,1)$.

In the group $H^{2g-1}_{\M}(X,\ZZ(g))$ there are certain {\bf decomposable} cycles coming from the product 
$$H^{2g-1}_{\M}(X,\ZZ(g))_{dec}= \Im \left( H^{2g-2}_{\M}(X,\ZZ(g-1) \otimes  H^1_{\M}(X,\ZZ(1))   ) \longrightarrow H^{2g-1}_{\M}(X,\ZZ(g))\right).$$ 
This is simply the image of $CH^{g-1}(X) \otimes \C^*$ - namely elements of the form $(C,a)$ where $C$ is a curve and $a$ is a constant function.  

The group of  {\bf indecomposable} cycles  is defined as the quotient 
$$H^{2g-1}_{\M}(X,\ZZ(g))_{ind}=H^{2g-1}_{\M}(X,\ZZ(g))/H^{2g-1}_{\M}(X,\ZZ(g))_{dec}.$$
In general it is not easy to find non trivial  elements in this group.

If $X$ is a surface, one way of constructing a possibly non-trivial cycle is the following. Since we will use it later we label it a proposition. 

\begin{prop} Suppose $C$ is a {\bf nodal rational}  curve on a surface $X$ with node $P$. Let $\pi:\tilde{X}\rightarrow X$ denote the blow up of $X$ at $P$ such that the strict transform $\tilde{C}$ of $C$ no longer has a node at $P$. Let $P_1$ and $P_2$ be the two points lying over $P$ on $\tilde{C}$.  If $E_P$ is the exceptional fibre it is a rational curve and it meets the strict transform at the points $P_1$ and $P_2$. Let $f$ be a function with $\div(f)=P_1-P_2$ on $\tilde{C}$ and let $g$ be a function on $E_P$ with divisor $\div (g)=P_2-P_1$. Then
	 $$\pi_*((\tilde{C},f)+(E_P,g))$$
	  is a cycle in $H^{3}_{\M}(X,\ZZ(2))$.
	
	\label{construction}
	
	\end{prop}

\subsection{Old cycles: Collino's cycle} \label{collino}

Let $C$ be a hyperelliptic curve of genus $g$ and $J(C)$ its Jacobian. Collino \cite{coll} constructed cycle in the group $H^{2g-1}_{\M}(J(C),\ZZ(g))=CH^{g}(J(C),1)$ as follows:

Let $P_1$ and $P_2$ be two Weierstrass points  or, equivalently, two of the ramification points of the map to $\CP^1$. Let $\TRR$ be a third point on $C$ distinct from $P_1$ and $P_2$.  There exists a function $f$ on $C$ with 
$$\div(f)=2P_1-2P_2$$
we further assume, though not strictly necessary for the purpose of constructing the cycle, that $\TRR$ is also a Weierstrass point and that $f(\TRR)=1$. 

Let $\iota_{P_i}:C \rightarrow J(C)$ be the map 
$$\iota_{P_i}(x)=x-P_i$$
Let $C_{P_i}=\iota_{P_i}(C)$. Since $P_1-P_2$ is a two torsion point on $J(C)$, $P_1-P_2=P_2-P_1$. One has  $C_{P_1} \cap C_{P_2}=\{O,P_1-P_2\}$ where $O$ is the identity in $J(C)$. Let $f_{P_i}$  be the function $f$ considered as a function on $C_{P_i}$. Then 
$$\div(f_{P_1})=2(0)-2(P_2-P_1) \text{ and } \div(f_{P_2})=2(P_1-P_2)-2(O)  $$
Hence $Z_{P_1,P_2,R}=(C_{P_1},f_{P_1})+(C_{P_2},f_{P_2})$ satisfies the cocycle condition and determines an element of the motivic cohomology group  $H^{2g-1}_{\M}(J(C),\ZZ(g)$.  

This cycle is defined as long as the hyperelliptic curve is irreducible. Collino \cite{coll}  shows futher  that it is generically indecomposable. In the case of genus $2$ the cycle is defined on the complement of a component of  the moduli of products of elliptic curves.

\section{New cycles on the moduli of Abelian sufaces}

\label{mainconstruction}

In this section we describe the construction of new cycles in the generic fibre of the universal Abelian surface over the Siegel modular threefold with full level $2$ structure. The idea is to use Proposition \ref{construction} -- but we do not have any rational curves on an Abelian variety. Hence we have to go to an auxilliary variety, the Kummer $K3$ surface. This construction is done in \cite{sree22-1} but we will recall it here. 

\subsection{Humbert Surfaces}

Over the Siegel modular threefold the Picard number of the generic Abelian surface is $1$. There are certain divisors corresponding to the moduli of those Abelian surfaces where the Picard number is at least $2$. These are called {\bf Noether-Lefshetz} divisors, or {\bf Heegner} divisors, or {\bf Humbert} surfaces as they were first studied by Humbert. 

Let $A=A_z$ be an Abelian surface  corresponding to a point $z$ on the Siegel modular threefold. Let $\theta$ be its  principal polarisation, so $\theta^2=2$. If $D$ is in $NS(A)$ define its {\bf Humbert invariant} 
$$\Delta(D)=(D.\theta)^2-2D^2$$
This is essentially the negative of the intersection pairing on the orthogonal complement of the class of $\theta$ in the Picard lattice $NS(A_z)$. In particular it is non-negative and is non-zero if and only if $D \notin \ZZ \cdot \theta$.  

Let $H_{\Delta}$  be the closure of the moduli of those Abelian surfaces in whose Picard lattice there is an element of Humbert invariant $\Delta$. $H_{\Delta}$ is called the {\bf Humbert Surface of invariant $\Delta$}. These is the same as Heegner divisors $H(\gamma,N)$ or Noether-Lefschetz divisors $D_{M,r}$ for some choices of $N,\gamma, M$ and $r$ \cite{pete}. Note that $H_1$ is the moduli of products of elliptic curves.

\subsection{The Theorem of Birkenhake and Wilhelm}

Birkenhake and Wilhelm \cite{biwi} determine conditions under which the moduli of an Abelian surface lies on $H_{\Delta}$ for certain $\Delta$. This is a generalisation of an old theorem of Humbert \cite{humb} for $\Delta=5$. The idea is the following: Given an Abelian surface $A=J(H)$ for some genus $2$ curve $H$,  the square of the principal polarization determines a map  $\phi:A \rightarrow \CP^3$. This is a double cover ramified at the $16$ two torsion points.  The image $K_A$ is called the {\bf Kummer surface} and is isomorphic to $A/\pm 1$. It is a quartic hypersurface in $\CP^3$.   Projecting from the image of $0$ gives a map  $\psi:K_A \rightarrow \CP^2$ and it turns out that this is a double cover ramified at six lines tangent to a conic. Let $\CP^2_A=(\CP^2,l_1,\dots,l_6)$ be the collection of $\CP^2$ with the six lines. Let $S_A=\bigcup_{i=1}^6 l_i$. This is a degenerate sextic with fifteen nodes corresponding to the images of the fifteen non-zero two torsion points. One has maps  
$$
A \stackrel{\phi}{\longrightarrow} K_A \stackrel{\psi}{\longrightarrow} \CP^2_A.$$
If one blows up $K_A$ at the sixteen points then one obtains a $K3$ surface $\TK_A$ which we will refer to as the {\bf Kummer $K3$ surface}. 

Hence to an Abelian surface, one can canonically associate a configuration of six lines in $\CP^2$ tangent to a conic. Conversely,  a degenerate sextic $S_A \subset \CP^2$ tangent to a conic determines a principally polarised Abelian surface. 

Given this correspondence one can ask what conditions on $\CP^2_A$ determines if the moduli point of the corresponding Abelian surface lies on $H_{\Delta}$ for some $\Delta$?.  Humbert \cite{humb} proved the following theorem. 

\begin{thm} Let $A$ be an Abelian surface and $\CP^2_A$ the corresponding configulation of $\CP^2$ with six lines $l_i$ tangent to a conic. Let $q_{ij}=l_i \cap l_j$ denote the fifteen nodal points of the degenerate sextic $S_A$. Then $A$ corresponds to a point on $H_{5}$ if and only if there exist a conic $\QQ$ passing through five of the points $q_{ij}$ and tangent to the remaining line. 
\end{thm}

Birkenhake and Wilhelm \cite{biwi}  generalised this to other $\Delta$. For certain classes of $\Delta=\Delta(d,k)$ determined by integers $k>2$ and $d>0$,  they show that the moduli point of an Abelian surface lies on $H_{\Delta}$ implies that there is a rational curve $\QQ$ of degree $d$ meeting the sextic $S_A$  at $k$  of the points $q_{ij}$ and the remaining lines at points of even multiplicity. Conversely, if there is such a rational curve, then it corresponds to some $\Delta' \leq \Delta$. 

To recover a divisor $D$ of invariant $\Delta>0$ from the rational curve,  we observe the following. 

\begin{prop} Let $\QQ$ be a rational curve on $\CP^2_A$ which meets the sextic $S_A$ only at points of even multiplicity or at the double points $q_{ij}$ as in the theorems of Birkenhake and Wilhelm. Let $\psi:\CC \longrightarrow Q$ be the double cover. Then $\CC$ is {\bf not} irreducible and is the union of two rational curves $\CC_1$ and $\CC_2$. 
	
	\label{rationalcurve}
\end{prop}

\begin{proof} The map $\CC \longrightarrow \QQ$ is ramified at the points of $\QQ \cap S_A$ which are all double points. Hence the ramification points are nodes of $\CC$. Let $\psi:\tilde{\CC} \longrightarrow \tilde{\QQ}$ be the maps between the normalizations induced by $\psi$. As all the ramification points of $\CC \rightarrow \QQ$ are singular points,  $\tilde{\CC}$  is an {\bf unramified} double cover of a $\CP^1$. Since the fundamental group of $\CP^1$ is trivial, there are no {\bf irreducible} unramified double covers.  Hence $\tilde{\CC}$ is the union of two rational curves meeting at a point and $\CC$ is the union of two curves $\CC_1$ and $\CC_2$ which meet at the points lying over the ramification points. 
	
\end{proof}

Let $\DD_i=\phi^*(\CC_i)$. Then $\DD_1$ (or $\DD_2$) will not be a multiple of the  $\theta$. Let $D=\DD_1$. Since $\DD_1+\DD_2=\phi^*(\psi^*(\QQ))$ is  multiple of $\theta$ it does not matter which one we choose. 

\subsection{A Theorem in Enumerative Geometry} 

We would like to deform the rational curve $\QQ$ that meets the configuration of six lines at double points to a rational curve that meets the configuration at fewer double points - but with the advantage that it always exists.  A classical theorem in enumerative geometry states that there exists  a unique conic passing through $5$ points in  general position in $\CP^2$. This has the following well known generalization. 

\begin{thm} Let $n_d$ be the number of {\em rational} curves of {\em degree $d$} passing through $3d-1$ points in general position. Then $n_d$ is finite and non-zero. 
		
\end{thm}

\begin{proof} \cite{mcsa}, Prop. 7.4.8, page 230. \end{proof}

The exact number $n_d$ was computed by Kontsevich-Manin \cite{koma} and Ruan-Tian \cite{ruti} in the early 90s. Classically it was known that  $n_1=1$, $n_2=1$ and $n_3=12$. More generally one has 

\begin{thm} Let $0 \leq k \leq 3d-1$. Let $n_{d,k}$ be the number of {\em rational curves} passing through $k$ points and meeting $(3d-1-k)$ lines tangentially. Then $n_{d,k}$ is finite and non-zero. 
	
	\end{thm}
	
For instance, $n_{2,0}=n_{2,5}=1, n_{2,1}=n_{2,4}=2,n_{2,3}=n_{2,2}=4,n_{3,7}=36$. The number $n_{d,k}$ is called a {\em characteristic number} and the analogue of the theorem of Kontsevich and Manin can be found in Pandharipande \cite{pand}.

We have the following theorem. 

\begin{thm}[Graber,S.] \label{enumgeom}
	
	Consider pairs $(S,\QQ)$ where $S=\bigcup_{i=1}^6 L_i$ is a degenerate sextic in $\CP^2$ given by a product of six lines and $\QQ$ is a rational curve of degree $d$. In general they will meet at $6d$ points. Suppose  there exists a degenerate sextic $S_0=\bigcup_i L_{0,i}$ and a rational curve $\QQ_0$ of degree $d$  such that they meet at $ \leq 3d+1$   points of the following type 
	
	\begin{itemize}
		\item $k$ nodes of the form $q_{0,ij}=L_{0,i} \cap L_{0,j}$ of $S_0$
		\item $3d-1-k$ other points tangent to the lines $L_{0,i}$ (which could coincide to be of even multiplicity). 
		\item $2$ other points (which could coincide and be of one  of the above types). 
	\end{itemize}
	
	Then, for {\bf any} degenerate sextic $S=\bigcup_i L_i$ the number $N_{d,k}$ of rational curves $\QQ$ of degree $d$ passing through the $k$ nodes and meeting the lines $L_i$  tangentially at $3d-1-k$ points is finite and non-zero.
	
	\end{thm}

\begin{proof} \cite{sree22-1}. \end{proof}

In fact, more generally if one has a family of sextics with $k$ sections corresponding to $k$ nodes one has a similar theorem. In particular, for smooth sextics if for some sextic there exists a rational curve of degree $d$ meeting it tangentially at $3d-1$ points then for any sextic there exists such a curve. 

Its an interesting question to count the exact number. For small $d$ the numbers are essentially  charcteristic numbers. For instance, if $d=2$ and $k=5$ this is simply $n_2=1$. When $d$ gets larger though we have to use our theorem as the number of lines is fixed.

\subsection{The motivic cycle}

In this section we show that there exists a motivic cycles in the group $H^3_{\M}(A_{\eta},\Q(2))$ where $A_{\eta}$ is the generic Abelian surface. These cycles $Z_{\Delta}$  are  defined in the fibres in the complement of components of $H_{\Delta}$. Further, they are  {\bf indecomposable}. Finally, like Collino's element the boundary of the cycles in the localization sequence is a multiple of the cycle $\DD_1-\DD_2$. where $\DD_1$ is an element of invariant  $\Delta$. We sketch the argument here. The details can be found in \cite{sree22-1}.

Let $A_z$ be an Abelian surface corresponding to a point $z$ on the moduli. This determines the lines $l_i(z)$, points $q_{ij}(z)$ and the degenerate sextic $S_A(z)=\bigcup l_i(z)$ in $\CP^2$ all of which  vary smoothly in $z$. Suppose for some $z_0$,  one has a rational curve $\QQ_0$ of degree $d$ as in Proposition \ref{rationalcurve}.  From Theorem of Birkenhake and Wilhem \cite{biwi} such $\QQ_0$ exists for infinitely many $(d,k)$.

A curve of degree $d$ meets a sextic at $6d$ points - however, since $\QQ_0$ meets $S(z_0)$ only at double points, there are at most $3d$ points. Let  $T(z_0):=S(z_0) \cap \QQ_0$. In $T(z_0)$ the points are of two types -- either points of the form $q_{ij}(z_0)$ or points $P_i(z_0)$ where the line $l_i(z_0)$ and  $\QQ_0$ meet with even multiplicity $m_{P_i}$. Let there be $k$ points of the type $q_{ij}(z_0)$. The set of all such $z_0$ is component $P_{\Delta}$ of a Humbert surface  $H_{\Delta}$ for some $\Delta$  determined by $\QQ_0$.  

From Theorems 7.1,7.2,7.3 and 7.4 of Birkenhake and Wilhelm \cite{biwi} at least $3$  points of $T(z_0)$  are of the type $q_{ij}(z_0)$.  Let $q=q_{i_0,j_0}(z_0)$ be one such point in  $T(z_0)$. For any $z$ there are  the points $q_{ij}(z)=l_i(z) \cap l_j(z)$. From Theorem \ref{enumgeom} there exists a rational curve $\QQ(z)$ such that 
\begin{itemize}
	\item $\QQ(z)$ is of degree $d$.
	\item $\QQ(z)$ passes though points of the form $q_{ij}(z)$ such that $q_{ij}(z_0) \in T(z_0)$ except $q$. 
	
	\item  $\QQ(z)$ meets the lines $l_i(z)$ with multplicity $m_{P_i}$.

	\item $\QQ(z)$ is such that $\QQ(z_0)=\QQ_0$. 
	
\end{itemize}

We have two cases:

\noindent {\bf Case 1:} If $z$ is not on  $P_{\Delta}$, the curve $\QQ(z)$ will meet $S(z)$ at two other points, say $s_{i_0}(z)$ and $s_{j_0}(z)$ on $l_{i_0}(z)$ and $l_{j_0}(z)$ respectively.
$T(z)=\QQ(z) \cap S(z)$  has $3d-1$ points which are of even multiplicity and $2$  points of multiplicity one.

Let $\CC(z) \rightarrow \QQ(z)$ be the double cover induced by the map $\psi:K_{A_z} \rightarrow \CP^2$. The normalization $\tilde{\CC}(z)$ is a double cover of $\CP^1$ which is ramified at the two points $s_{i_0}(z)$ and $s_{j_0}(z)$ and is hence an {\bf irreducible smooth rational curve} and $\CC(z)$ is a {\bf nodal rational curve}. The nodes are at the points lying over $T(z)$. 

\noindent {\bf Case 2:} If  $z$ lies on  $P_{\Delta}$ the points $s_{i_0}(z)$ and $s_{j_0}(z)$ coincide at the point $q_{i_0j_0}(z)$.  $T(z)$ has $3d$ points of even multplicity. From Proposition \ref{rationalcurve} one has that $\CC(z)$ is the union of two curves $\CC_1(z)$ and $\CC_2(z)$. These curves meet at the points lying over $T(z)$.

Since all the objects vary smoothly with $z$ one has a family of rational curves $\CC(z)$ on $K_{A_z}$ which break up in to two components when $z$ lies on $P_{\Delta}$. 

Now assume $z$ does not lie on $P_{\Delta}$. Recall that the $K3$ surface $\hat{K}_{A_z}$ is obtained by blowing up the $16$ double points of $K_{A_z}$ - which are the points lying over $q_{ij}(z)$ and the image of $0$. Let $\pi:\hat{K}_{A_z} \longrightarrow A_z$ be the birational map given by the blow-up. 

In $T(z)$ there are  at least two points  of the form $q_{ij}(z)$. Let $\hat{\CC}(z)$ be the strict transform of $\CC(z)$. Let $P(z)=q_{ij}(z)$ and $R(z)=q_{i'j'}(z)$ be two of the points in $T(z)$ and $P_1(z)$, $P_2(z)$, $R_1(z)$ and $R_2(z)$ the points lying over them in the exceptional fibres over $P(z)$ and $R(z)$ in $\hat{K}_{A_z}$ respectively. 

Let $E_{P(z)}$ be the exceptional fibre over $P(z)$.  $P_1(z)$ and $P_2(z)$ lie in $\hat{\CC}(z) \cap E_{P(z)}$.  let $f_{P(z)}$ be the function on $\hat{\CC}(z)$ with  divisor 
$$\div (f_{P(z)})=P_1(z)-P_2(z)$$
and such that  $f_{P(z)}(R_1(z))=1$,   Let $g_{P(z)}$ be a function on $E_{P(z)}$ with divisor 
$$\div(g_{P(z)})=P_2(z)-P_1(z).$$
Such functions exists as both $\hat{\CC}(z)$ and $E_{P(z)}$ are {\em rational} curves. We have the following theorem:

\begin{thm} For $z \notin P_{\Delta}$, let 
	$$\tZ_{\QQ_0}(z)=(\hat{\CC}(z),f_{P(z)}) + (E_{P(z)},g_{P(z)})$$
	Then $\tZ_{\QQ_0}(z)$ is an element of the group $H^3_{\M}(\hat{K}_{A_z},\Q(2))$. Further, it defines an element of the motivic cohomology group of the generic fibre $H^3_{\M}(\hat{K}_{A_{\eta}},\Q(2))$ which is {\bf indecomposable} and has boundary  a non-zero multiple of $\hat{\CC}_1(z)-\hat{\CC}_2(z)$ in the fibres over $P_{\Delta}$, up to the boundary of a decomposable element. 
  
		\end{thm}
	\begin{proof} \cite{sree22-1}\end{proof}
	
	We can pull this back to the Abelian surface $A_z$ to get a motivic cycle $\Xi_{\Delta,r}(z)$ in $H^3_{\M}(A_z,\Q(2))$ and this cycle is defined outside a component of $H_{\Delta}$. Note that Collino's cycle in the case when $C$ is a genus $2$ hyperelliptic curve is defined in the complement of a component of $H_1$.

	\subsubsection{An Example: Humbert's theorem}
	
	As an example of the above theorem one can consider the following situation. Let $l_i(z)$ be the six lines of a degenerate sextic and $q_{ij}(z)=l_i(z) \cap l_j(z)$ corresponding to an Abelian surface $A_z$. Recall Humbert's theorem  states the following. There is a conic $Q(z)$ passing through five points $q_ij(z)$ and meeting the remaining line tangentially if and only if $\End(A(z)) \simeq \ZZ[\sqrt{5}]$ or equivalently, $z\in H_{5}$. $H_{5}$ has six components in the moduli of Abelian surfaces with level $2$ structure and the different components correspond to different choices of lines. 
	
	For instance we can consider the points $q_{12}(z),q_{23}(z),q_{34}(z),q_{45}(z)$ and $q_{51}(z)$ and the line $l_6(z)$. Then if there exists a conic $\QQ(z)$ passing through these five points and tangent to $l_6(z)$ the corresponding $A_z$ has extra endomorphisms.
	
	However, given {\em any} five points in general position there exists a conic passing though them. So for any $z$ there is a conic $\QQ(z)$ passing through $q_{12}(z),\dots,q_{51}(z)$. In general, though,  it will not be tangent to $l_6(z)$ - it will meet it at $2$ points $s_1(z)$ and $s_2(z)$. 
	
	The normalization of the double cover of $\QQ(z)$ is a double cover of $\CP^1$ ramified at two points - the points lying over $s_1(z)$ and $s_2(z)$. This is an irreducible conic. The image of that in $K_{A_z}$ is an irreducible rational curve with nodes at the points $q_{12}(z),\dots,q_{51}(z)$. 
	
	To build the motivic cycle we consider the blow up of this conic in the $K3$ surface $\TK_{A_z}$ and follow the procedure above. There are six components of $H_5$ corresponding to which of the six lines the exceptional conic is tangent to. This gives a motivic cycle defined in the complement of the component corresponding to the line $l_6$.

\section{Old and new cycles}

In this section we relate the Collino cycle to  our new cycle. We first need some generalities. 

\subsection{The Universal Property of the Jacobian}

If $D$ is a curve, its Jacobian $J(D)$ satisfies the following universal property. 

\begin{prop} Let $D$ be a curve of genus $g$  and $P$ a point on $D$.  Let $\eta: D \rightarrow A$  be a map from $D$ to an Abelian variety $A$ such that $\eta(P)=0$.  Then there is a unique {\bf homomorphism} $\tilde{\eta}:J(D) \longrightarrow A$ such that 
		$$\eta=\tilde{\eta} \circ \iota_P$$
where $\iota_P:D \longrightarrow J(C)$ is the map 
$$\iota_P(x)=x-P$$
\label{univjacobian}
\end{prop}

\begin{proof} (Sketch). Define $\tilde{\eta}:D^g \longrightarrow A$ by 
	$$\tilde{\eta}(Q_1,\dots,Q_g)=\sum \eta(Q_i)$$
	This is clearly invariant under the symmetric group and hence descends to a rational map, which we also call $\tilde{\eta}:J(D) \longrightarrow A$. Further, this can be seen to be a morphism which sends $0$ to $0$ and is hence a homomorphism. The map $\iota_P:D \longrightarrow J(D)$ is the composite of the maps $D \rightarrow D^g \rightarrow J(D)$ given by $Q \rightarrow (Q,P,\dots, P)$ and $(Q_1,\dots,Q_g) \rightarrow \sum_i \iota_P(Q_i)$. Hence $\eta=\tilde{\eta} \circ \iota_P$. 
	\end{proof} 

We have following corollary which drops the assumption that $\eta(P)=0$. The result is that one does not have a homomorphism from $J(C)\rightarrow A$, only a morphism. 

\begin{cor}
	Let $D$ be a curve of genus $g$ and  $\mu:D \rightarrow A $ be a map such that $\mu(D)=Q$ for some point $Q$ on $A$. Then there is a morphism $\tilde{\mu}:J(C) \rightarrow A$ such that
	$$\mu=\tilde{\mu} \circ \iota_P$$
	Further, if $P'$ is another point such that $\mu(P')=Q$, then $\mu=\tilde{\mu} \circ \iota_{P'}$ as well. 
	\label{univjaccor}
\end{cor}

\begin{proof} Let $\eta=T_{-Q} \circ \mu$ where $T_{*}:A \rightarrow A$ is the translation map $T_*(x)=x+*$. Then $\eta:C \rightarrow A$ satisfies $\eta(P)=0$ and we can apply Proposition  \ref{univjacobian}. Hence there is a homomorphism $\tilde{\eta}:J(C) \rightarrow A$ such that $T_{-Q} \circ \mu=\eta=\tilde{\mu} \circ \iota_P$. Let $\tilde{\mu}=T_Q \circ \tilde{\eta}$. One then has 
	$$\mu=T_{Q} \circ \eta= T_{Q} \circ \tilde{\mu} \circ \iota_P=\tilde{\mu} \circ \iota_P$$
	Since the construction of $\tilde{\eta}$ involves only $\eta$ and not $P$ it does not depend on the choice of $P$ in $\mu^{-1}(Q)$. 
	
	\end{proof}

Let $\QQ_0$ be fixed and $z$  in the complement of the corresponding $P_{\Delta}$. Consider the  cycle $\Xi=\Xi_{\QQ_0}(z)=\phi^*(\pi_*(\tZ_{\QQ_0}(z)))$ which is the pull-back of the cycle $\tZ_{\QQ_0}(z)$ to the Abelian surface $A_z$.  This is an element of $H^3_{\M}(A_z,\Q(2))$.   In what follows since we have fixed $z$ we will drop it from our notation.

\subsubsection{Nodal hyperelliptic curves}

The cycle $\Xi$  is of the form $(\DD,h_{P})$  where $\DD$ is a {\bf nodal hyperelliptic curve}, namely a double cover of the nodal rational curve $\CC$ on $K_A$. The point $P$ is a node as well as a ramification point of the double cover. 


Let $\nu:\tilde{\DD} \rightarrow \DD$ denote the normalization of $\DD$. Since $\DD$ is a singular hyperelliptic curve, $\tilde{\DD}$ is a {\bf smooth hyperelliptic curve}. Let $\TCC$ be the normalization of $\CC$. From the universal property of the normalization there is a unique map $\tilde{\phi}:\tilde{\DD} \longrightarrow \TCC$ such that the following diagram commutes 
$$\begin{CD}
	\tilde{\DD} @>\tilde{\phi}>> \TCC\\
	@VV\nu V  @VV\nu V\\
	\DD@>\phi>> \CC
\end{CD}
$$
$\tilde{\DD}$ is the double cover of the normalization of $\CC$. Recall that we have the points $P_1$ and $P_2$ lying over the node $P$ in $\CC$. 

\begin{lem}
	The points lying over the node $P$ under the map $\tilde{\DD} \rightarrow \DD$ are ramification points of the map $\tilde{\phi}:\TDD\rightarrow \TCC$. To reduce notation, we will continue to call them $P_1,P_2$. 
\end{lem}
\begin{proof} One has an involution $i$ on $A$ induced by the double cover $\phi:A \rightarrow K_A$ which restricts to give the hyperelliptic involution on $\DD$. The point $P$ on $\CC$ is fixed by $i$. In the blow up of $K_A$ at $P$  every point in the  exceptional fibre $E_P$ is fixed by $i$ and hence so are the points $P_1$ and $P_2$. Hence they are ramification points of the hyperelliptic curve $\TDD \rightarrow \TCC$. 
	
	\end{proof}

For the same reason, the points $R_1$ and $R_2$ are Weierstrass points.  In particular, there exists a function $f_{P_1P_2}$ on $\TDD$ with 
$$\div(f_{P_1P_2})=2P_1-2P_2 \text{ and } f_{P_1P_2}(R_1)=1$$

\subsubsection{Collino's cycle revisited}

Since $P_1$,$P_2$ and $R_1$  are Weierstrass points  in $J(\TDD)$, following Collino,  we can use them to construct a cycle in $H^{2g-1}_{\M}(J(\TDD),\Q(g))$. Let $Z_{P_1P_2,R_1}$ be that cycle. We have the following theorem which relates it to the cycle we have constructed above. 

\begin{thm} Let $Z_{P_1P_2,R_1}$ be the Collino cycle corresponding to $P_1$ and $P_2$, where $R_1$ is a Weierstrass point mapping to the point $R$ on $\DD$. Then 
$$\tilde{\mu}_*(Z_{P_1P_2,R'})=2 \Xi_{P,R}$$
where $\tilde{\mu}$ is the map $J(\TDD) \longrightarrow A$ obtained by  Corollary \ref{univjaccor}. 

\end{thm}

\begin{proof}
	From Corollary \ref{univjaccor} The map $\mu:\TDD\rightarrow \DD$ determines a map $\mu:\TDD \rightarrow A$ which takes $P_1$   to $P$ and detemines a morphism $\tilde{\mu}:J(\TDD) \longrightarrow A$ such that 
	$\mu=\tilde{\mu} \circ \iota_{P_1}$. Under this map the $(\TDD_{P_1},f_{P_1})$ is taken to $\Xi_{P,R}$  Howvever, the map also takes $P_2$ to $P$ and one has $\mu=\tilde{\mu} \circ \iota_{P_2}$. Hence it takes $(\TDD_{P_2},f_{P_2})$ is also take to $\Xi_{P,R}$ as well. 

	%
	So the image of $C_{P_2}$ is the same as $C_{P_1}$. Hence both the cycles which make up $Z_{P_1P_2,R}$ map to the same cycle  $\Xi_{P,R}$ and we have 
	$$\tilde{\mu}_*(Z_{P_1P_2,R})=2 \Xi_{P,R}$$
\end{proof}

This shows that in fact the cycle we constuct in \cite{sree22-1} is not entirely a new cycle but a version of Collino's cycle for a different hyperelliptic curve, and in a sense goes even further back to the work of Bloch \cite{bloc} where he constructed a cycle on $X_0(37) \times X_0(37)$ which maps to the Collino cycle under the map $X_0(37) \times X_0(37) \rightarrow J(X_0(37)$. 

This result suggests that it is not so easy to find new motivic cycles - all constructions seem to be a variation on the original construction of Bloch. In \cite{sree22-1} we speculated on a relationship between weakly holomorphic modular forms and motivic cycles on Abelian surfaces and it would be interesting to see if there is some fundamental modular form such that all modular forms are derived from it in some way. 

A positive consequence of this result is that it is known that the Collino cycle is natural in the following sense. From generalities one expects that motivic cycles can be understood as certain extensions in the category of mixed motives - thought this is still conjectural. However, the regulator of this cycle can be understood as a extension in the category of Mixed Hodge structures. In Colombo \cite{colo} and more generally in Sarkar-Sreekantan \cite{sasr} it was shown that the regulator of Collino's cycle or Bloch's cycle can be realised in terms of natural extensions of mixed Hodge structures coming from the {\em fundamental group} of the curve.  

Thanks to this connection with the Collino cycle we can obtain the extension class associated to our new motivic cycle. 

\begin{cor} Recall that there is an extension class $\bar{e}^4_{P_1P_2,R}$ in $\Ext_{MHS}(\ZZ(-g),H^{2g-2}(J(\TDD)))$ corresponding to Collino's cycle $Z_{P_1P_2,R_1}$. From the theorem above, the class  $\frac{1}{2}\tilde{\mu}_*(\bar{e}^4_{P_1P_2,R})$  is the extension class corresponding to the cycle $\Xi_{P,R}$ in $\Ext_{MHS}(\ZZ(-2),H^2(A))$. 
		
	\end{cor}

\bibliographystyle{alpha}
\bibliography{Extensions.bib}

\begin{tabular}[t]{l@{\extracolsep{8em}}l} 
	Ramesh Sreekantan \\
	Statistics and Mathematics Department\\
	Indian Statistical Institute \\
	8th Mile, Mysore Road \\
	Jnanabharathi\\
	Bengaluru 560 059 \\ 
	rameshsreekantan@gmail.com 
\end{tabular}

\end{document}